\documentclass[11pt]{article}
\usepackage{amsfonts,amsmath,amssymb,enumerate,amsthm}

\usepackage{graphicx}
\usepackage{color}
\usepackage{color}
\usepackage{dsfont}
\definecolor{nb}{rgb}{.6,.176,1}
\definecolor{sienna}{rgb}{1,0,0}
\definecolor{darkgreen}{rgb}{0,.5,0}

\usepackage{fullpage}

 \newtheorem{theorem}{Theorem}[section]
 \newtheorem{proposition}[theorem]{Proposition}
 \newtheorem{corollary}[theorem]{Corollary}
 \newtheorem{conjecture}[theorem]{Conjecture}
 \newtheorem{lemma}[theorem]{Lemma}

 \newtheorem{remark}{Remark}
 \newtheorem{defi}[theorem]{Definition}

 \newcommand{\bt}{\begin{theorem}}
 \newcommand{\et}{\end{theorem}}
 \newcommand{\bl}{\begin{lemma}}
 \newcommand{\el}{\end{lemma}}
 \newcommand{\bp}{\begin{proposition}}
 \newcommand{\ep}{\end{proposition}}
 \newcommand{\bcor}{\begin{corollary}}
 \newcommand{\ecor}{\end{corollary}}
 \newcommand{\br}{\begin{remark}\rm}
 \newcommand{\er}{\end{remark}}
 \newcommand{\bcon}{\begin{conjecture}}
 \newcommand{\econ}{\end{conjecture}}
 \newcommand{\bd}{\begin{defi}}
 \newcommand{\ed}{\end{defi}}
 \newcommand{\ben}{\begin{enumerate}}
 \newcommand{\een}{\end{enumerate}}

 \newcommand{\Ebb}{\mbox{${\mathbb E}$}}

 \newcommand{\GG}{\mbox{${\mathcal G}$}}
\newcommand{\Zbold}{\mbox{${\mathbb Z}$}}
 \newcommand{\Nbold}{\mbox{${\mathbb N}$}}

  \newcommand{\eps}{\epsilon}

 \newcommand{\T}{{\mathbb T}}
 
 \renewcommand{\P}{\mathbb P}
 \newcommand{\E}{{\mathbb E}}
\newcommand{\Ccal}   {{\mathcal C }} 
\newcommand{\FF}{\mbox{${\mathcal F}$}}

 \usepackage{color}
\usepackage{color}
\usepackage{dsfont}

\definecolor{nb}{rgb}{.6,.176,1}
\definecolor{sienna}{rgb}{1,0,0}
\definecolor{darkgreen}{rgb}{0,.5,0}

\newcommand{\bone}{{\mathbf 1}}
\def\1{I}
\renewcommand{\SS}{\mbox{${\mathcal S}$}}
\newcommand{\bP}{{\bf P}}
 \newcommand{\bE}{{\bf E}}

 \newcommand{\Pbb}{\mbox{${\mathbb P}$}}

\def\1{I} 

\def\keywords{\vspace{.5em}
{\textit{Keywords}:\,\relax%
}}

\begin{document}

 \renewcommand{\labelenumi}{{(\roman{enumi})}}

 \title{{\bf S.L.L.N. and C.L.T.  for Random Walks in I.I.D. Random Environment on 
  Cayley Trees}}

\date{}
\author{
Siva Athreya\footnote{Theoretical Statistics and Mathematics Unit, Indian Statistical Institute, 8th Mile Mysore Road, Bangalore,
560059, INDIA, athreya@isibang.ac.in}
\and 
Antar Bandyopadhyay\footnote{Theoretical Statistics and Mathematics Unit, Indian Statistical Institute, Delhi Centre, 7 S. J. S. Sansanwal Marg, New Delhi 110016, INDIA; and 
Theoretical Statistics and Mathematics Unit, Indian Statistical Institute, Kolkata, 203 B. T. Road, Kolkata 700108, INDIA, antar@isid.ac.in}
\and
Amites Dasgupta\footnote{Theoretical Statistics and Mathematics Unit, Indian Statistical Institute, Kolkata, 203 B. T. Road, Kolkata 700108, INDIA, amites@isical.ac.in}
\and
Neeraja Sahasrabudhe\footnote{Department of Mathematical Sciences, Indian Institute of Science Education and Research, Mohali, Knowledge city, Sector 81, Manauli, Sahibzada Ajit Singh Nagar, Punjab 140306, INDIA, neeraja@iisermohali.ac.in}
}

  \maketitle

\begin{abstract}
\noindent We consider the random walk in an  {independent and identically distributed (i.i.d.)}  random environment on a Cayley graph of a finite \emph{free product} of copies of $\Zbold$ and $\Zbold_2$. Such a Cayley graph is readily seen to be a regular tree.  Under a uniform elipticity assumption on the i.i.d. environment we show that the walk has positive speed and establish the annealed central limit theorem for the  graph distance of the walker from the starting point.
\end{abstract}


\section{Introduction}
\label{Sec:Intro}

In this article, we consider a \emph{random walk in random environment
  (RWRE)} model on a regular tree, which was introduced in
\cite{ABD13}. Like in any other (static) RWRE model, in our model also, we
first choose an \emph{environment} by some 
random mechanism and keep it fixed throughout the time evolution. 
A walker then moves randomly on the vertex set of a regular tree
in such a way that given the environment, its
position forms a time homogeneous Markov chain whose transition
probabilities depend only on the environment.  RWRE model on the one
dimensional integer lattice ${\mathbb Z}$ was first introduced by
Solomon in \cite{Sol75} where he gave explicit criteria for the
recurrence and transience of the walk for \emph{independent and
  identically distributed (i.i.d.)}  environment distribution.
Perhaps the earliest known results for RWRE on trees is by Pemantle
and Lyons \cite{LyPe92}, where they consider a model on
rooted tress, which later got to known as \emph{random conductance
  model}. In their model, the random conductances along each path from
vertices to the root are assumed to be independent and identically
distributed. The random walk is then shown to be recurrent or
transient depending on how large is the value of the average
conductance. Motivated by these, \cite{ABD13}, considered a RWRE model on a
regular tree, where the environment (or rather the transition laws)
at each vertex are
independent and also ``identically'' distributed. 
However, unlike in the usual RWRE models on integer lattices, such as 
on $\Zbold$ as introduced by \cite{Sol75}, or the random conductance models
on trees \cite{LyPe92}, it is not entirely obvious how to 
make the random transition laws on the vertices of a tree ``identically'' distributed.
To make this notion
of i.i.d.-ness of the environment rigorous in \cite{ABD13} defined 
the model on a
the finite free product of copies of $\Zbold$ and $\Zbold_2$
 and then transfers it back
to an appropriate degree regular tree which is essentially same
as a Cayley graph associated with the group. A more detailed description
is given in the following subsection. 
A similar model was also considered in \cite{Rozi01}.

In both  \cite{Rozi01, ABD13} under differing mild non-degeneracy assumptions the authors
proved that the RWRE is transient with probability one. In this work,
we extend their result and prove that under \emph{uniform ellipticity}
assumption on the i.i.d. environment (which is stronger assumption
than the one made in \cite{ABD13}), the walk has a positive drift 
away from the starting point and
admits an \emph{annealed Central Limit Theorem} under linear centring
and square-root scaling. Our proofs are motivated by the work
\cite{ABD13}.

\subsection{ Model and Main Results}
\label{SubSec:Results}

Even though our framework is same as in \cite{ABD13}, but for sake of 
completeness, 
we begin by providing a detailed description of the model below. \\

\noindent
{\bf Group structure:} 
Following \cite{ABD13} we will also consider a group $G$ which is a free
product of finitely many groups, say, $G_1, G_2, \cdots, G_k$ and
$H_1, H_2, \cdots, H_r$, where 
each $G_i \cong \Zbold$ and each $H_j \cong \Zbold_2$. Let $d = 2k + r$.

\medskip

\noindent
{\bf Cayley graph:} Let $G$ be a group defined above. 
Suppose $G_i = \langle a_i \rangle$ for $1 \leq i \leq k$ and 
$H_j = \langle b_j \rangle$ where $b_j^2 = e$ for $1 \leq j \leq r$.
Here by $\langle a \rangle$ we mean the group generated by a single element $a$.  
Let 
$S := \left\{a_1, a_2, \ldots, a_k\right\} \cup \left\{a_1^{-1}, a_2^{-1}, \ldots, a_k^{-1}\right\}
      \cup \left\{ b_1, b_2, \ldots, b_r \right\}$
be a generating set for $G$. We note that $S$ is a symmetric set, that is, $s \in S \iff s^{-1} \in S$. 

We now define a graph $\bar{G}$ with vertex set $G$ and edge set 
$E := \left\{ \left\{x,y\right\} \,\Big\vert\, yx^{-1} \in S \,\right\}$. We will say $x \sim y$ whenever $\{x,y\} \in E$.
Such a 
graph $\bar{G}$ is called a \emph{(left) Cayley Graph} of $G$ with respect to the generating set $S$. 
Since $G$ is a free product of groups which are isomorphic to either $\Zbold$ or 
$\Zbold_2$, 
it is easy to see
that $\bar{G}$ is a graph with no cycles and is regular with degree $d$, thus it is isomorphic to 
the $d$-regular infinite tree which we will denote by $\T_d$. 
We will abuse the terminology a bit and will write $\T_d$ for 
the Cayley graph of the group $G$. This way we essentially endow the 
$d$-regular tree, $\T_d$ a \emph{group structure}, which we will make use to 
define an \emph{i.i.d.} environment. 

Note that for the $d$-dimensional Euclidean lattice, such a group
structure is automatic, which is the product of $d$ copies of the
abelian free group $\Zbold$. In our case, all the difference appears
due to the fact that on $\T_d$, a group can only be obtained through
free product of several copies of $\Zbold$ and also with possible free
product of groups generated by \emph{torsion} elements.

We will consider the identity element $e$ of $G$ as the root of $\T_d$.
We will write $N\left(x\right) := \left\{ y \in G \,\Big\vert\,  yx^{-1} \in S \,\right\}$ 
for the set of all neighbors of a vertex $x \in \T_d$. 

Observe that from definition 
$N\left(e\right) = S$.

For $x \in G$, define the mapping $\theta_x : G \rightarrow G$ by $\theta_x\left(y\right) = yx$, 
then $\theta_x$ is an automorphism of $\T_d$. We will call
$\theta_x$ the \emph{translation by} $x$. For a vertex $x \in \T_d$ and $x \neq e$, 
we denote by $\left\vert x \right\vert$, 
the length of the unique path from the root $e$ to $x$ and $\left\vert e \right\vert = 0$. Further,
if $x \in \T_d$ and $x \neq e$ then we define $\overleftarrow{x}$ as the 
\emph{parent} of $x$, that is, the penultimate vertex on the unique path from
$e$ to $x$.

\medskip

\noindent
{\bf Random Environment:} Let $\SS := \SS_{e}$ be a  
collection of probability
measures  on the $d$ elements of $N\left(e\right) = S$. 
To simplify the presentation and avoid various measurability issues,
we assume that $\SS$ is a Polish space (including the possibilities
that $\SS$ is finite or countably infinite).
For each $x \in \T_d$,  $\SS_{x}$ is the the push-forward of the space
$\SS$  under the translation $\theta_x$, that is, 
$\SS_x := \SS \circ \theta_x^{-1}$. 
Note that the probabilities on $\SS_x$ have support on
$N\left(x\right)$. That is to say, an element
$\omega(x,\cdot)$ of $\SS_{x}$, is a probability measure
satisfying
$\omega\left(x, y\right) \geq 0 \,\, \forall \,\, y \in \T_d$ and
$ \mathop{\sum}\limits_{y \in N\left(x\right)} \omega\left(x, y\right) = 1$.

Let ${\mathcal B}_{{\mathcal S}_x}$ denote the Borel $\sigma$-algebra on $\SS_x$. The 
\emph{environment space} is defined as the measurable space $\left(\Omega, \FF\right)$ where

$\Omega := \mathop{\prod}\limits_{x \in \T_d} \SS_{x}$ and

$\FF := \mathop{\bigotimes}\limits_{x \in \T_d} {\mathcal B}_{{\mathcal S}_x}$.

An element $\omega \in \Omega$ will be written as
$\left\{ \omega\left(x, \cdot\right) \,\Big\vert\, x \in \T_d \,\right\}$.
An environment distribution is a probability $\P$ on $\left(\Omega, \FF\right)$. 
We will denote by $E$ the expectation taken with respect to the probability measure $\P$.

\medskip

\noindent
{\bf Random Walk:} Given an environment 
$\omega \in \Omega$, a random walk $\left(X_n\right)_{n \geq 0}$
is a time homogeneous Markov chain taking values in $\T_d$ 
with transition probabilities given by
$\left(\omega\left(x, y\right)\right)_{x, y \in \T_d}$.
Let $\Nbold_0 := \Nbold \cup \left\{ 0 \right\}$. 
For each $\omega \in \Omega$, we denote by
$\bP_{\omega}^{x}$ the law induced by $\left(X_n\right)_{n \geq 0}$ on
$\left( \left(\T_d\right)^{\Nbold_0}, \GG \right)$,
where $\GG$ is the $\sigma$-algebra generated by the cylinder sets, 
such that $\bP_{\omega}^{x}\left( X_0 = x \right) = 1.$  The probability measure $\bP_{\omega}^{x}$ is called the \emph{quenched law}
of the random walk $\left(X_n\right)_{n \geq 0}$, starting at $x$. We will use the notation $\bE_\omega^x$ for the expectation under the quenched measure $\bP_\omega^x$. 

Following Zeitouni \cite{Zei04}, we note that for every $B \in \GG$, the function
$\omega \mapsto \bP_{\omega}^{x}\left(B\right)$ is $\FF$-measurable. Hence, we may define the measure
$\Pbb^{x}$ on 
$\left( \Omega \times \left(\T_d\right)^{\Nbold_0},
\FF \otimes \GG \right)$ by the relation
\[
\Pbb^{x}\left( A \times B \right)
= 
\int_A \! \bP_{\omega}^{x}\left(B\right) \P\left(d\omega\right), 
\,\,\,\, \forall \,\,\, A \in \FF, \, B \in \GG. 
\]
With a slight abuse of notation, we also denote the marginal of 
$\Pbb^{x}$ on $\left(\T_d\right)^{{\mathbb N}_0}$
by $\Pbb^{x}$, whenever no confusion occurs. This probability distribution
is called
the \emph{annealed law} of the random walk $\left(X_n\right)_{n \geq 0}$, 
starting at $x$.  We will use the notation $\Ebb^x$ for the expectation under the annealed measure $\Pbb^x$

{\bf Assumptions:} Throughout this paper we will assume the following hold,
\begin{itemize}
\item[(E1)] $\P$ is a product measure on $\left(\Omega, \FF\right)$ with ``\emph{identical}'' marginals, that is, 
            under $\P$ the random probability laws $\left\{ \omega\left(x, \cdot\right) \,\Big\vert\, x \in \T_d \,\right\}$
            are independent and ``identically'' distributed in the sense that
            \begin{equation} \label{si} \P \circ \theta_{x}^{-1} = \P, \end{equation}
            for all $x \in G$. 
            
\item[(E2)] There exists $\eps > 0$ such that 
            \begin{equation} 
            \label{Equ:UEllip} 
            \P\left( \omega\left(e, s_i\right) > \eps \,\, \forall \,\, 1 \leq i \leq d \right) = 1.
            \end{equation}

\end{itemize}

We are now ready to state our main results. We begin with a law of
large numbers result for $\mid X_n \mid$ which also establishes that the
speed of walk is positive.
 
\begin{theorem} 
\label{speed} 
Assume (E1) and (E2). Then there exists $v >0$,  such that, 
\begin{align}
\lim_{n \rightarrow \infty}\frac{| X_n |}{n} =  v,   
\end{align}
almost surely with respect to $\P^{e}$
\end{theorem}

 In  \cite{ABD13}, it was pointed out that under (E1), (E2),
$\liminf_{n \rightarrow \infty} \frac{\left\vert X_n \right\vert}{n} > 0$, 
if  $\eps > \frac{1}{2\left(d-1\right)}$. 
The above result not only establishes that the walk on $\T_d$ has
a \emph{positive speed}, but also shows that the corresponding limit
exits almost surely. Our next result is an annealed central limit theorem for 
$\mid X_n \mid$.

\begin{theorem} 
\label{CLT}
 Assume (E1)and (E2). Then there exists $\sigma^2 > 0$ such that, 
 under $\P^{e}$,
\begin{equation}
\sqrt{n} \left( \frac{|X_n|}{n} - v \right) \stackrel{d}{\rightarrow}
Z,
\end{equation} 
with $Z ~ \mbox{Normal}(0, \sigma^2)$.
\end{theorem}

We note here that although we define the walk starting at $X_0 = e$,
 the root, results hold for starting at any vertex $x$
 of $\T_d$. This is because the environment is invariant under the translation
 by the group $G$.  Indeed it will be evident from the proofs that the
 constants $v$ and $\sigma^2$ are also independent of the starting
 position. Thus the Theorems ~\ref{speed} and ~\ref{CLT} hold for any
 initial distribution of $X_0$ on the vertex set of $\T_d$.

The basic framework of proof is inspired by the arguments laid out in
\cite{CS11}. Formally speaking we begin by defining regeneration times
as first time the walk reaches a new level and never visits it
again. We then prove moment bounds for these regeneration times and
followed established techniques laid out in the RWRE literature to
obtain our results. However, we note that the nature of i.i.d in the
environment (assumption (E1)) in this paper is derived from the group
structure of the graph and is a different from that of the environment
in \cite{CS11}. So, even though the basic framework was available we
needed different techniques for executing the same.

\section{Regeneration Times}
\label{Sec:Regeneration}

In this section we shall introduce a sequence of \emph{regeneration
  times} and provide moment bounds for them. We begin with some
notation.  For any $x, y \in \T_d$, recall that  $ [x,y] = \{ \{x_i\}_{i=0}^n \mid x_0 = x, x_n = y, x_i^{-1}x_{i-1} \in S, 1 \leq i \leq n \}.$  Let $T_d(y)$ be the sub-tree rooted at $y$, i.e $\{ x \in \T_d: y \in [e,x] \}$ and $\T^n_d =\{x \in \T_d: \mid x \mid = n \}$. The type of $x \in \T_d, x \neq e$ is $s\in S$ if $\overleftarrow{x}^{-1} x = s$ and we shall denote it by $s_x$.

\begin{equation}
\label{T}
T(y) := \inf \{ n \ge 0 \colon X_{n} = y\} 
\end{equation}
and
\begin{equation}
 R(y) := \inf\{ n \ge 1 \colon X_{n-1}\in \T_d(y), \, X_n= y\},
 \label{R}
\end{equation}
be the hitting time of $y$ and the return time to $y$, respectively.
We also define, 
\begin{equation}
\label{beta0}
\T_n := \inf \{ k \ge 0 \colon X_k \in \T^n_d \} 
\end{equation}
and
\begin{equation}
R := \inf\{ n \ge 1 \colon X_{n} = X_0\},
\end{equation}
to be the hitting time of level $n$ and return time of the walk to its
starting point, respectively. 

The \emph{first regeneration level } is then defined as
$l_1 := \inf \{ k \geq 1: R(X_{T_k}) = \infty \}$,
and the
\emph{$n-\mbox{th}$ regeneration level} for $n \geq 2$ is defined 
recursively as 
$l_n := \inf \{ k \geq l_{n-1}: R(X_{T_k}) = \infty \}$.
Regeneration times for $n \geq 1$, are defined by
\begin{equation} 
\label{regt}
\tau_n := \left\{ \begin{array}{lcl}
T_{l_n} & \,\, & \mbox{if\ } l_n < \infty; \\
            &       &                                       \\
\infty   &  \,\, & \mbox{otherwise}.
\end{array} \right.
\end{equation}

\subsection{Tail bounds for the first regeneration time}
\label{SubSec:Tail-Bounds}
We begin by proving tail bounds for first regeneration
level, followed by moment bounds on number of visits of the walk to
the root and number of distinct vertices visited before first
regeneration. We conclude this section with the required moment bounds
on the regeneration times defined above (See Proposition~\ref{tau1}).

Let $\{h_n(x,y) |n \geq 1, x, y \in \T_d, \mbox{ and } x \sim y \}$ be
i.i.d Exponential random variables with mean $1$. Suppose $X_0 =x$, it is easy to verify that 
$\displaystyle X_{n+1} = \underset{y \sim X_n}{\mbox{argmin}} \frac{h_{n+1}(X_n, y)}{\omega(X_n, y)}$.
Fix a finite sub-tree $\Ccal$ of $\T_d$ with $x \in \Ccal$.  Let
$\{Y^{\Ccal}_n\}_{n \geq 0}$ be such that $Y^{\Ccal}_0 = x$ and
$Y^{\Ccal}_{n+1} = \underset{y \sim Y^{\Ccal}_n, y \in
  \Ccal}{\mbox{argmin}} \frac{h_{n+1}(Y^{\Ccal}_n,
  y)}{\omega(Y^{\Ccal}_n, y)}$. It is easy to see that
$\{Y^{\Ccal}_n\}_{n \geq 0}$ is a Markov chain on $\Ccal.$ We shall
construct a coupling with the walk $\{X_n: n \geq 1\}$ so that
$Y^{\Ccal}$ has the same law as $X$ whenever it visits the sub-tree
$\Ccal.$ Define
$\lambda_1 = \inf \{ n \geq 0 : X_n \notin \Ccal \}$,
and recursively for $i \geq 1$,
$$\mu_{i} = \inf \{ n \geq \lambda_i : X_n \in \Ccal \} \mbox{
  followed by } \lambda_{i+1} = \inf \{ n \geq \mu_i : X_n \notin
\Ccal \}.$$ 
Define for each $k \geq 1$,
\[ 
W_k = 
\begin{cases} X_k & \mbox{ if } k \leq \lambda_1-1;\\ 
X_{\mu_{j} + k} & \mbox{ if } \mu_j <  k \leq \lambda_{j}-1, 
\mbox{for some } j \geq 1.
\end{cases} 
\]
Note that as $\T_d$ is a tree, for all $i \geq 1,$
$X_{\lambda_{i}-1} = X_{\mu_i} \in \Ccal.$ Further $X_{\mu_{j} + k} \in \Ccal$ if  $\mu_j <  k \leq \lambda_{j}-1$ for some  $j \geq 1$. It is easy to see that $\{W_n\}_{n \geq 0}$ is a Markov chain on $\Ccal$ and has the same law as $\{Y^{\Ccal}_n\}_{n \geq 0}.$

\medskip
\noindent
\textbf{Colouring scheme :} \rm We begin by colouring the root $e$ as
red. Let $k \geq 1$ and $\psi \geq 1$. A vertex $y \in \T^{k\psi}_d$ is coloured red if and only if 
\begin{itemize}
\item its ancestor at level $(k-1)\psi$, say $x$, is coloured red, and 
\item $\{Y_n^{[x,y]}\}_{n \geq 0}$, started at $x$, hits $y$ before
  returning to $x$.
\end{itemize}
For each $\psi \geq 1$ and $k \geq 1$, $s \in S$ let $Z_\psi(k,s)$ be the number of red vertices at level $k \psi$ of type $s$. Let $Z_\psi(0) = \{ e \}$ and for $k \geq 1$, define $$Z_\psi(k) :=  \{Z_\psi(k,s) :{s \in S}\}.$$
Under the annealed measure, $\{ Z_\psi\}$ is a multi-type Branching process with expected offspring matrix $M= (m_{Sue})_{s\in S, u\in S}$ is given by  
\[ 
m_{su} = \E \left[  \sum_{x_\psi  \in \T^\psi_d} \left ( \sum_{m=1}^{\psi-1} \prod_{j=1}^m \frac{\omega(x_j, x_{j-1})}{\omega(x_j, x_{j+1})} \right )^{-1} \right], \]
where $s, u \in S  \mbox{ and } s_{x_1} = s, \mbox{ and }  s_{x_\psi} = u.$ 

\begin{proposition} \label{BP1} Assume (E1) and (E2). There exists $\psi \geq 1$ such that the $Z_\psi$ is supercritical.
\end{proposition}
\begin{proof} We will show that the largest eigenvalue, $\rho,$ of the offspring matrix $M$ is larger than $1$. We observe that for $1\leq i \leq n-1$
\[ \mathbf{P}_{\omega}(Y_n^{[x,y]}= x_{i-1} \mid Y^{[x,y]}= x_{i}) = \frac{\omega(x_i, x_{i-1})}{\omega(x_i, x_{i+1}) + \omega(x_i, x_{i-1})}\]
and
\[\mathbf{P}_{\omega}(Y_n^{[x,y]}= x_{i+1} \mid Y^{[x,y]}= x_{i}) = \frac{\omega(x_i, x_{i+1})}{\omega(x_i, x_{i+1}) + \omega(x_i, x_{i-1})}.\]
Using a standard gambler's ruin chain argument we can conclude
\[ \,\mathbf{P}_{\omega}(Y^{[x,y]} \mbox{ hits $y$ before returning to $x$ }) = \left ( \sum_{m=0}^{n-1} \prod_{j=0}^m \frac{\omega(x_j, x_{j-1})}{\omega(x_j, x_{j+1})} \right )^{-1}.\]
From the arguments in proof of Theorem 1 in  \cite{ABD13}, it is easy to see that,  for $1 < c_1 < d-1$,  there exists $n  \geq 1$, such that,
\[ 
\frac{1}{d(d-1)^{n-1}}  \# \left \{ \sigma_n \in \T^n_d : \left ( \sum_{m=1}^{n-1} \prod_{j=1}^m \frac{\omega(x_j, x_{j-1})}{\omega(x_j, x_{j+1})} \right )^{-1} \geq c_1^{-(n-1)} \right \} \geq \frac{1}{2} \mbox{\ a.s.\ } \P.
\]
Therefore, for all $s \in S$, and large enough $\psi$
\begin{eqnarray}
 \sum_{u \in S} m_{su} = \E \left[  \sum_{x_\psi \in \T^\psi_d} \left ( \sum_{m=1}^{\psi-1} \prod_{j=1}^m \frac{\omega(x_j, x_{j-1})}{\omega(x_j, x_{j+1})} \right )^{-1} \right ] \geq \frac{d(d-1)^{\psi-1}}{2}c_1^{-(\psi- 1)} > 1.
\label{Equ:Row-Sum-Larger-than-1}
\end{eqnarray}
As the row sums of $M$ are larger than $1$, this implies that the largest eigenvalue $\rho $ is bigger than $1$  and this implies the process is super-critical.

\end{proof}

Now, for $x \in \T_d$,  $y \in \T_d(x) $ is called a {\it first child}  if 
it is (almost surely) the minimiser of 
\begin{equation}
\label{firstch}
\min\limits_{z \sim x, z \neq \overleftarrow{x} }\frac{h_{1}(x,z)}{\omega(x,z)}
\end{equation}
For $m \geq 1$, 
\[ F_{m,x} = \{ y \in \T_d(y): \mid y \mid - \mid x\mid = m  \mbox{ and } y \mbox{ is a {\it first child}} \}. \] 
Let $\psi \geq 1, \zeta \geq 1$
\[ \Sigma_x = \T_d(x)\cap Z_{\psi} \cap_{k=1}^\infty F^c_{k\zeta\psi,x}\]

\[ B(x)= \{ \Sigma_{x} \mbox{ is finite}\},B_0 = B(e), \mbox{ and } B_{k} = B(X_{T_{k \psi \zeta }}), \,\,k \ge 1.\]

\begin{lemma} 
\label{independent}
The collection of events $\{B_i, i \geq 1\}$, are independent.
\end{lemma}
\begin{proof} The event $ B(x) \in \sigma \{ h_n(z,y): z, y
\in \T_d(x) n \geq 1 \}$.  Let $i_1 < i_2 < \ldots < i_k$
be positive integers. Note, as observed,
$$ B_{i_j} \in \sigma \{ h_n(z,y) :z, y \in \T_d({X_{i_j \zeta \psi}}) n \geq 1 \}.$$ Note however that $X_{i_j \zeta \psi}$ is a first child at level $i_j \zeta \psi.$ and this implies $ \T_d({X_{i_j \zeta \psi}})\cap  \T_d({X_{i_l \zeta \psi}})= \emptyset.$  Hence  $\{B_{i_j} : 1 \leq j \leq k\}$ are mutually independent.
\end{proof}

\begin{proposition} 
\label{l1bound}
$\exists \,\, \gamma < 1$ such that for $n \geq 1$, we have $$\P(l_1 \geq n \psi \zeta) \leq \gamma^{n-1}.$$
\end{proposition}

\begin{proof} Note that $B^c_i \subseteq \{ \mbox{level} \ i \psi \zeta \ \mbox{is a regeneration level} \}$. Hence, using that $B_i$ are independent we have,
\begin{equation*}
\P( l_1 \geq n \psi \zeta  ) \leq \P(\bigcap\limits_{i=1}^{n-1} B_i) = \prod\limits_{i=1}^{n-1} \P(B_i).
\end{equation*}
For $s \in S$, let $$ B^s_i :=  \{\overleftarrow{X}^{-1}_{T_{i \zeta \psi}} X_{T_{i \zeta \psi}} = s\} \cap B_i. $$ Note that for $s \in S$,  $\P(B^s_i) = \P(B^s_j)$ for all $1 \leq i,j \leq n.$  Hence , $$\P(B_i) =  \P(\cup_{s \in S} B_i^s) = \sum_{s \in S} \P(B_i^s) = \gamma$$ Therefore
$\P( l_1 \geq n \psi \zeta  ) \leq \prod_{i=1}^{n-1} \left ( \sum_{s \in S} \P(B_i^s) \right) = \gamma^{n-1}$.

Now we will show that we can choose $\zeta > 0$, such that, $\gamma < 1$. 
It is enough to show that we can choose $\zeta > 0$, such that, 
$\P(B_1) < 1$. For this we follow an argument similar to the proof of
Lemma 3.3 of \cite{CS11}. From definition, it is clear that the vertices which belong to $\Sigma_{X_{\psi \zeta}}$, 
are obtained as follows. The vertices at level $(\zeta - 1) \psi$ are of $d$-types, 
and has a distribution with mean matrix $M^{(\zeta - 1) \psi}$. Further, the vertices
at level $\zeta \psi$ has a number of various types of coloured vertices, and we have 
deleted the first child, thus the expectation matrix of such vertices is $M-A$, where 
$A$ is a $d \times d$-matrix with $0 \leq  A_{su} \leq 1$ and $A \bone = \bone$. 
Thus, 
$\gamma = \P(B_1)$ is at most as large as, the extinction probability of a multi-type  branching process with mean matrix $\tilde{M}_{\zeta} := M^{(\zeta - 1) \psi} (M-A)$. 
But, from equation ~\eqref{Equ:Row-Sum-Larger-than-1}, it follows that 
\[
m_0 := \min_{s \in S} \sum_{u \in S} m_{s u} > 1.
\]
So for any $s \in S$, the $s$-th row sum of $\tilde{M}_{\zeta}$ is
at least as large as $m_0^{(\zeta - 1) \psi} (m_0 - 1)$. Now select 
$\zeta  \geq 1$, such that, $m_0^{(\zeta - 1) \psi} (m_0 - 1) > 1$. From 
the argument above then we can conclude that $\gamma < 1$. 
\end{proof}

\subsection{Moment bounds}
For $x \in \T_d,$ let $T(x)$ be as in (\ref{T}) and $\tau_1$ be as in (\ref{regt}). Further let 
\begin{equation}
\label{eq:L}
 L(x) := \sum_{j=0}^{\infty} \1_{\{X_{j} = x\}} \mbox { and } {\cal D} := \sum\limits_{x\in \T_d} \1_{ \{ T(x) \leq \tau_1 \} }  
\end{equation}
be total number visits to $x$ and  the number of distinct vertices visited before $\tau_1$ respectively.

\begin{proposition} Assume (E1) and (E2). Then for $p \geq 1$,
 \label{tau1}
  \begin{enumerate}
    \item[(a)] $\E[L(e)^p] < \infty$;
    \item[(b)]$\E[{\cal D}^p] < \infty$; and
    \item [(c)] $\E[\tau_1^p]<\infty$.
      \end{enumerate}
\end{proposition}

\begin{proof} (a) For $n \geq 1$, let  $$\mathcal{U}_n = \{\{x^i\}_{i=1}^d : x_i \in \T^n_d  \mbox{ and } [x^i,e] \cap [x^j,e] = \{e\} \mbox{ for all } 1 \leq i \neq j \leq d\}.$$ We will denote any element of $U_n$ by $\mathcal{A}_n$ and the smallest sub-tree in $\T_d$ containing $\mathcal{A}_n$ will be denoted by $\mathfrak{T}_n$. Consider the walk $\{Y^{\mathfrak{T}_n}_k : k \geq 1 \}$. Define
$$\tilde{T}_{\mathcal{A}_n} = \inf \{ k \geq 1: Y^{\mathfrak{T}_n}_k \in \mathcal{A}_n \}, \mbox{ and }   \tilde{L}(e, \tilde{T}_{\mathcal{A}_n}) = \sum_{i=0}^\infty \1_{\{ Y^{\mathfrak{T}_n}_i = e, i < \tilde{T}_{\mathcal{A}_n} \}}$$
to be the hitting time of ${\mathcal A}_n$ and  the number of visits of $Y^{\mathfrak{T}_n}$ to $e$ before the walk $Y^{\mathfrak{T}_n}$ hits $\mathcal{A}_n$ respectively. Define  $\tilde{R}_n = \inf \{ k \geq 1: Y^{\mathfrak{T}_n}_k = e\}$ return time  to $e$. Under the quenched law, $ \tilde{L}(e, \tilde{T}_{\mathcal{A}_n})$ is a geometric random variable with parameter $q_\omega$ given by
\begin{eqnarray*}
q_w &=& {\mathbf P}_\omega (\tilde{T}_{\mathcal{A}_n} < \tilde{R}_n) = \sum\limits_{i=1}^d w(e, s_i) \left( \sum\limits_{j=1}^n \prod_{k=1}^{j-1} \frac{\omega(x^i_k, x^i_{k-1})}{\omega(x^i_k, x^i_{k+1})} \right)^{-1}.
\end{eqnarray*}

Using standard results about geometric random variables we know that for $p > 1$, $\exists \,\, c_p >0$, such that ,
\begin{eqnarray} \label{tildeLbound}
\E[\tilde{L}(e, \tilde{T}_{\mathcal{A}_n})^p ] & \leq& c_p \E[q_\omega^{-p}] = c_p\E \left[ \left(\sum\limits_{i=1}^d w(e, s_i) \left( \sum\limits_{j=1}^n \prod_{k=1}^{j-1} \frac{\omega(x^i_k, x^i_{k-1})}{\omega(x^i_k, x^i_{k+1})} \right)^{-1}\right)^{-p} \right] \nonumber \\
 & \leq &  c_p\, d^{-p}\,\E \left[ \left( \min \limits_{1 \leq i \leq d} \left( \sum\limits_{j=1}^n \prod_{k=1}^{j-1} \frac{\omega(x^i_k, x^i_{k-1})}{\omega(x^i_k, x^i_{k+1})} \right)^{-1} \right)^{-p} \right] \nonumber \\
 & = &  c_p\, d^{-p}\,\E \left[ \max\limits_{1\leq i \leq d} \left(  \sum\limits_{j=1}^n \prod_{k=1}^{j-1} \frac{\omega(x^i_k, x^i_{k-1})}{\omega(x^i_k, x^i_{k+1})} \right)^{p} \right].
\end{eqnarray}

For $x \in \T_d$, define $ R^x := \inf \{n \geq 1: Y^{\T_d(x}_n)= x \}$, and
\begin{eqnarray*}
 H& := & \inf \{k \geq 1:  \exists  \mathcal{A}_k \in \mathcal{U}_k  \mbox{ such that } {\mathbf P}^x_{\omega}(R^x = \infty) = 1  \mbox{ for all } x \in A_k\}. 
 \end{eqnarray*}
Observe that for any $n \geq 1$,
$$ L(e) I_{\{H=n\}} \leq  \tilde{L}(e, \tilde{T}_{\mathcal{A}_n}) I_{\{ H = n \}}. $$

By   H\"older's inequality  for $\epsilon>0$, with $a = 1 + \epsilon/p,$ $b = 1 + p/\epsilon$, and using (\ref{tildeLbound}) we have 
 \begin{eqnarray} \label{istlb}
   \E[L(e)^p, I_{\{ H = n \}}] &\leq & \E[ (\tilde{L}(e, \tilde{T}_{\mathcal{A}_n}))^p I_{\{ H = n \}}]
\nonumber   \\ &\leq&  \E[ \tilde{L}(e, \tilde{T}_{\mathcal{A}_n})^{pa}]^{1/a} \P(H = n)^{1/b} \nonumber \\
&\leq & \left( c_{pa}\, d^{-pa}\,\E \left[ \max\limits_{1\leq i \leq d} \left(  \sum\limits_{j=1}^n \prod_{k=1}^{j-1} \frac{\omega(x^i_k, x^i_{k-1})}{\omega(x^i_k, x^i_{k+1})} \right)^{pa} \right]\right)^{1/a} \P(H = n)^{1/b}. \nonumber\\
 \end{eqnarray}
 Now,
\begin{eqnarray} \label{mbound}
    \P(H=n) & \leq& \P \left ( \bigcup\limits_{s \sim e} \bigcap\limits_{y \in \T_d(s)\cap \T^{n-1}_d}\{ R^y < \infty\} \right) \leq \sum\limits_{s \sim e} \P \left (  \bigcap\limits_{y \in \T_d(s)\cap \T^{n-1}_d}\{ R^y < \infty\} \right) \nonumber \\ &=& \sum\limits_{s \sim e}  \prod\limits_{y \in \T_d(s)\cap \T^{n-1}_d}\P \left ( R^y < \infty \right)  , \leq  d \left(\max\limits_{s \in S, y \in \T_d(s)\cap \T^{n-1}_d}\P \left (s_y = s,  R^y < \infty \right) \right)^{(d-1)^{n-2}}. \nonumber \\
 \end{eqnarray}
Let $q (s) = \P(R^{y}< \infty : y \in \T^{1}_d, s_{y}= s)$ and $q = \max_{s \in S} q(s)$. From~\eqref{istlb} and \eqref{mbound} we have,
 \begin{eqnarray} \label{Lbound}
\E[L(e)^p, I_{\{ H = n \}}]  & \leq  & \left( c_{pa}\, d^{-pa}\,\E \left[ \max\limits_{1\leq i \leq d} \left(  \sum\limits_{j=1}^n \prod_{k=1}^{j-1} \frac{\omega(x^i_k, x^i_{k-1})}{\omega(x^i_k, x^i_{k+1})} \right)^{pa} \right]\right)^{1/a}  \left (d q^{(d-1)^{n-2}} \right)^{\frac{1}{b}} \nonumber \\ 
& \leq & c_1  c_2^n q^{\frac{(d-1)^{n-2}}{b}} 
 \end{eqnarray}
 Now it easily follow that $\E[L(e)^p]= \sum_{n=1}^\infty \E[L(e)^p, I_{\{ H = n \}}] < \infty.$

(b) From the definition of ${\cal D}$, we have
\begin{eqnarray*}
  {\cal D} &=&   \sum\limits_{x\in \T_d} \1_{ \{ T(x) \leq \tau_1 \} } =  1 + \sum_{x \neq e, x \in \T_d} \sum_{n=1}^\infty  \1_{\{T(x) \le T_n\}}\,\1_{\{l_1=n\}} \\
  &=& 1 + \sum_{n=1}^\infty \sum_{x \neq e, x \in \T_d}   \1_{\{T(x) \le T_n\}}\,\1_{\{l_1=n\}}\\
    &=& 1 + \sum_{n=1}^\infty \left ( \sum_{k=1}^{n} \sum_{x \in \T^k_d}\1_{\{ T(x) \le T_n\}}\,\right )\1_{\{l_1=n\}}\\
    &\leq& 1 + \sum_{n=1}^\infty \left (\sum_{k=1}^{n}  \sum_{x \in \T^k_d} \1_{\{T(x) < \infty \}}\,\right) \1_{\{l_1=n\}}.\\
\end{eqnarray*}
For each $k \geq 1$, we may dominate the random variable  $\sum_{x \in \T^k_d} \1_{\{T(x) < \infty \}}$ by a Geometric ($1-q$) random variable (where $q$ was defined in the previous proof). This implies
$$ \E \left (\sum_{x \in \T^k_d} \1_{\{T(x) < \infty \}}\,\right)^{2p} \leq c_p (1-q)^{-2p}.$$ Using this, Jensen's inequality,  followed by H\"older's inequality we have for all $n \geq 1$,
\begin{eqnarray*} \E  \left (\left (\sum_{x \in \T^k_d} \1_{\{T(x) < \infty \}}\,\right)^p \1_{\{l_1=n\}}  \right)  &\leq&  n^{p-1} \sum_{k=1}^{n}  \E  \left (\left (\sum_{x \in \T^k_d} \1_{\{T(x) < \infty \}}\,\right)^p \1_{\{l_1=n\}}  \right) \\ &\leq &n^{p-1} \sum_{k=1}^{n}  \sqrt{\E \left (\sum_{x \in \T^k_d} \1_{\{T(x) < \infty \}}\,\right)^{2p} }\sqrt{\P(l_1=n)}\\
  &\leq & c_p (1-q)^{-2p}  n^{p} \sqrt{\P(l_1=n)}
  \end{eqnarray*}
 Then, 

\begin{equation} 
 \E[{\cal D}^p]  \le  c_p (1-\P(R< \infty))^{-p} \sum_{n=1}^{\infty} n^p \,\P(l_1 = n)^{1/2}.
 \end{equation}
The result follows from Proposition \ref{l1bound}.

(c) Let $\{x_i : 1 \leq i \leq {\cal D}\}$ be an  enumeration of the vertices visited by the walk $X$ before time $\tau_1$.  It is easy to see that
$$ \tau_1 = \sum_{i=1}^{{\cal D}} L(x_i).$$
So,
\begin{equation}\label{tpr1}
  \E[\tau_1^p] = \E \left[ \left( \sum\limits_{i=1}^{\cal D} L(x_i) \right)^p \right]   \leq  \E \left[  {\cal D}^{p-1} \sum\limits_{i=1}^{\cal D} L(x_i)^p  \right].
  \end{equation}
  Now $$  {\cal D}^{p-1} \sum\limits_{i=1}^{\cal D} L(x_i)^p  =  \sum\limits_{i=1}^\infty {\cal D}^{p-1} L(x_i)^p I_{\{ {\cal D} \geq i \}}. $$
  For each $i \geq 1$, using H\"older's inequality twice (first with with  $q=1+\delta/p$, and $q^\prime=1+p/\delta$) and Chebychev's inequality
  , we have
  \begin{eqnarray}
    \E\left[ {\cal D}^{p-1} L(x_i)^p I_{\{ {\cal D} \geq i \}} \right] &\leq& \left[  \E[ L(x_i)^{p+\delta}] \right]^{1/q} \left[ \E[{\cal D}^{(p-1)q^\prime}I_{\{{\cal D} \geq i\}}] \right]^{1/(q^\prime)} \nonumber \\
    %&&\mbox{applying H\"older's inequality again we have}  \nonumber\\
    &\leq&  \left[ \E[ L(x_i)^{p+\delta}]\right]^{1/q} \left[ \E[{\cal D}^{2(p-1)q^\prime}] \right]^{1/(2q^\prime)}  \P({\cal D} \geq i)^{1/(2q^\prime)}  \nonumber\\
    %&&\mbox{applying Chebychev's inequality  we have}  \nonumber\\
    & \leq&  \left[ \E[ L(x_i)^{p+\delta}] \right]^{1/q} \left[ \E[{\cal D}^{2(p-1)q^\prime}] \right]^{1/(2q^\prime)}  \left[  \E({\cal D}^{4q^\prime}) \right]^{1/(2q^\prime)} \frac{1}{i^2}. \nonumber \\
  && \label{tpr2}
  \end{eqnarray}
By definition of $L$ we have
  \begin{eqnarray}
\E[L(x_i)^{p+\delta}] & \leq &\E[L(x_i)^{p+\delta}) | X_0 = x_i] = \E[L(e)^{p+ \delta}]. \label{tpr3}
\end{eqnarray}
Using (\ref{tpr1}), (\ref{tpr2}), (\ref{tpr3}), the fact that $\sum_{i=1}^\infty \frac{1}{i^2} = \frac{\pi^2}{6}$, along with part (a) and (b), we have the result.
\end{proof}

\section{Proof of Main Results}

\begin{proof}[{Proof of Theorem \ref{speed}:}] Let $\{\tau_n\}_{n \geq 1}$ be the sequence of regeneration times defined in (\ref{regt}). By Proposition (\ref{tau1}), $\E(\tau_1) < \infty$. So for all $n \geq 1$, there exists a (random) subsequence $\{k_n\}_{n\geq 1}$ such that 
\begin{equation}\label{knta}
\tau_{k_n} < n \leq \tau_{k_n + 1}.
\end{equation}
It is then readily seen that
\begin{equation} \label{spt1}  
\frac{\mid X_n\mid }{n} = \frac{  \mid X_{\tau_1} \mid + \sum_{l=1}^{k_n-1} (\mid X_{\tau_{l+1}}\mid - \mid X_{\tau_{l}}\mid) + \mid X_n \mid -  \mid X_{\tau_{k_n}}\mid}{ \tau_1 + \sum_{l=1}^{k_n-1} (\tau_{l+1} - \tau_{l}) + n -  \tau_{k_n}}. 
\end{equation}
For any $s \in S$, define
\begin{equation} \label{yz}
 Y_i(s) =  \begin{cases}\tau_1 \1_{\{s_{X_{\tau_1}} =s\} } & i = 1\\ 
    (\tau_i -\tau_{i-1}) \1_{\{s_{X_{\tau_i}=s}\}} & i > 1 \end{cases}, \quad  Z_i(s) =  \begin{cases}\mid X_{\tau_1} \mid I_{\{s_{X_{\tau_1}} =s\} } & i = 1\\ 
    (\mid X_{\tau_i} \mid  - \mid X_{\tau_{i-1}} \mid) \1_{\{s_{X_{\tau_i}=s}\}} & i > 1 \end{cases}.
  \end{equation}
Then, for each $s \in S$, $ \{Y_i(s)\}_{i \geq 1}$  and  $ \{Z_i(s)\}_{i \geq 1}$  are i.i.d. Using Proposition \ref{tau1}, 
$$\E[Y_i(s)]= \E[Y_1(s)] = \E[\tau_1 \1_{\{s_{X_{\tau_1}=s}\}} ] \leq \E[\tau_1]  < \infty, $$
and
$$\E[Z_i(s)]= \E[Z_1(s)] = \E[\mid X_{\tau_1} \mid \1_{\{s_{X_{\tau_1}=s}\}} ] \leq \E[\mid X_{\tau_1} \mid] \leq c_1 \E [\tau_1]  < \infty. $$
By strong of law of large numbers for each $s \in S$,
$$
  \lim_{n \rightarrow \infty} \frac{\sum_{i=1}^n Y_i(s)}{n} \rightarrow \E[Y_1(s)] \mbox{ and } \lim_{n \rightarrow \infty} \frac{\sum_{i=1}^n Z_i(s)}{n} \longrightarrow \E[Z_1(s)]
  $$
almost surely $\P$, as $n \rightarrow \infty$. Consequently,
\begin{equation}\label{nlimy}\frac{\sum_{l=1}^{k_n -1} \tau_{l+1} - \tau_l} {k_n-1} = \sum_{s \in S} \sum_{i=1}^{k_n-1} \frac{Y_i(s)}{k_n -1} \longrightarrow \sum_{s\in S}  \E[Y_1(s)] =\E[\tau_1],\end{equation}
  and 
\begin{equation}\label{nlimz}\frac{\sum_{l=1}^{k_n -1} \mid X_{\tau_{l+1}}\mid - \mid X_{\tau_l}\mid} {k_n-1} = \sum_{s \in S} \sum_{i=1}^{k_n-1} \frac{Z_i(s)}{k_n -1} \longrightarrow \sum_{s\in S} \E[Z_1(s)] = \E[\mid X_{\tau_1}\mid],
\end{equation} 
almost surely $\P$,  as $n \rightarrow \infty$. Now,
\begin{eqnarray}
  &&\label{ineq1} 0 \leq n - \tau_{k_n} \leq \tau_{k_n +1} - \tau_{k_n},\\
 && \E [\tau_{k_n} -\tau_{k_n-1}] = \sum_{s\in S} \E[\tau_{k_n}
   -\tau_{k_n-1} \1_{ \{s_{X_{\tau_{k_n}}}=s \} }] \nonumber \\&&= \sum_{s\in S} \E
   [\1_{\{s_{X_{\tau_{k_n}}}=s\}} \E[\tau_{k_n} -\tau_{k_n -1} \mid s_{X_{\tau_{k_n}}}=s ]] =
   \sum_{s \in S} \E [\1_{\{s_{X_{\tau_{k_n}}}=s\}} \E[\tau_2 -\tau_{1} \mid
   s_{X_{\tau_1}}=s ]]   \nonumber\\ &&\leq   \sum_{s\in S}\E [\E[\tau_2 -\tau_{1} \mid
         s_{X_{\tau_1}}=s ]] = \mid S \mid \E[\tau_2- \tau_1], \label{ineq2}\\
  &&\mbox{ and } \nonumber \\
     &&  \mbox{ $0 \leq \mid X_n \mid  - \mid X_{\tau_{k_n}} \mid \leq \mid X_{\tau_{k_n +1}} \mid - \mid X_{\tau_{k_n}} \mid$ and $n  - \mid X_{\tau_{k_n}} \mid \leq c_1 (n - \tau_{k_n})$} \label{ineq3}
\end{eqnarray}
 for some $c_1>0$. Using (\ref{nlimy})-(\ref{ineq3}) along with simple elementary algebra on (\ref{spt1})  yields
$$ \frac{\mid X_n\mid }{n} \longrightarrow \frac{\E[\mid X_{\tau_1}\mid]}{\E[\tau_1]}$$
almost surely  as $n \rightarrow \infty$.
\end{proof}

\bigskip

\begin{proof}[{Proof of Theorem \ref{CLT}:}]
Recall from  (\ref{yz}) and (\ref{knta}), $k_n, \tau_{\cdot}$, $Z_{\cdot}(\cdot)$ and $Y_{\cdot}(\cdot)$. Let $$ v =  \frac{\E[X_{\tau_1}]}{\E[\tau_1]}, \,\,  W_k(s)= Z_{k}(s) -Y_{k}(s) v, \, \, S_n(s) =  \sum\limits_{k = 1}^{n} W_k(s), \mbox{ and } \, S_n = \sum_{s\in S} S_n(s).$$
Observe that,
$$ \sqrt{n} \left( \frac{|X_n|}{n} - v \right) = \sqrt{n} \left( \frac{|X_n|}{n} - S_{k_n}- v \right) + \frac{S_{k_n}}{\sqrt{n}}.$$
As,
$$\frac{1}{\sqrt{n}} \mid  \, \, \mid X_n \mid - S_{k_n} - nv\mid  \leq \max_{1\leq i \leq k_n} \frac{\tau_i-\tau_{i-1}}{\sqrt{n}},$$
our result will follow if we establish that as $n \rightarrow \infty$
\begin{equation} \label{cclt} 
\frac{S_{k_n}  }{\sqrt{n}} 
\stackrel{d}{\rightarrow} N(0,\sigma^2), \,\, \mbox{ for some } \sigma^2 >0
\end{equation}
and for all $\delta >0$
\begin{equation} \label{step0} 
\P \left(
\max_{0\leq i\leq k_n}
\frac{\tau_{i+1}-\tau_i}{\sqrt{n}} >\delta \right) \longrightarrow 0.
\end{equation}

\noindent 
{\em Proof of (\ref{cclt}): } Note that it is easy to check that $s_{X_{\tau_i}}$ is uniform on $S$ (see \cite[Section 2]{ABD13})and thus the vector $(W_k(s))_{s \in S}$ form an i.i.d sequence with
  \begin{eqnarray*}
     \E[W_1(s)]&=& \E[X_{\tau_1}I_{\{s_{X_{\tau_1}=s}\}}] - \E[\tau_1I_{\{s_{X_{\tau_1}=s}\}}]v \mbox{ for each $s \in S$ and}\\
     \sigma_{s_1 s_2} &=&  \E \prod_{i=1}^2[\left (\mid X_{\tau_1} \mid I_{\{s_{X_{\tau_1}=s_i}\}} - \tau_1I_{\{s_{X_{\tau_1}=s_i}\}} v -\E[W_1(s_i)] \right )] \mbox{ for $s_1, s_2 \in S$.}
  \end{eqnarray*}
Therefore by the Multivariate central limit theorem, we have  
\begin{eqnarray*}
  \frac{(S_n(s)  -n \E[W_1(s)])_{s \in S}}{\sqrt{ n}} \stackrel{d}{\rightarrow} N(0,\Sigma)  
\end{eqnarray*}
where $\Sigma = (\sigma_{s_i s_j})_{s_i, s_j \in S}.$  The continuous map theorem then implies that for each $s \in S$,
\begin{eqnarray*}
  \frac{S_n(s)  -n \E[W_1(s)]}{\sqrt{ n}} \stackrel{d}{\rightarrow} N(0,\sigma_s^2)  
\end{eqnarray*}
with $\sigma^2_s =\sigma_{ss}$. If   $\sigma = 1^t\Sigma 1$ then it is immediate that as $n \rightarrow \infty$
 \begin{equation}\label{ccclt}\frac{S_n  }{\sqrt{ n}} \stackrel{d}{\rightarrow} N(0,\sigma^2). \end{equation}    Note that there is no centering
because $\sum_{s \in S} \E[W_1(s)] = 0. $ From proof of Theorem \ref{speed}, we know that  $\frac{k_n}{n} \rightarrow \frac{1}{\E[\tau_1]}.$  Using this and (\ref{ccclt}) we are done.\\

\noindent 
{\em Proof of (\ref{step0}):} Using Proposition \ref{tau1} (c) the proof  is standard (See \cite[Proof of Theorem 2.3]{BZ06}).
Since $k_{n}\leq n$, we have that for
any $\delta>0$, \begin{equation} \label{step1}   \P
\left( \max_{0\leq i\leq k_{n}}
\frac{\tau_{i+1}-\tau_i}{\sqrt{n}} >\delta \right) \leq
\sum_{i=1}^{n}  \P(\tau_1>\delta \sqrt
n).  
\end{equation} Note that, since $
\E\left[\tau_1^2\right]<\infty$, one has that $$
\sum_{i=1}^{\infty}  \P (\tau_1>\frac{\delta
\sqrt{i}}{\sqrt T})= \sum_{i=1}^\infty  \P
(\tau_1^2>\frac{\delta^2 i}{T})<\infty\,.$$ Hence, for each
$\epsilon>0$ there is a deterministic constant $N \equiv N(d,\delta,\epsilon)$ such that $$
\sum_{i=N}^{\infty}  \P (\tau_1>\frac{\delta
  \sqrt{i}}{\sqrt{T}})< \epsilon\,.$$
Therefore, \begin{equation*} 
\limsup_{n\to\infty}\sum_{i=1}^{n}
 \P(\tau_1>\delta \sqrt n)   \leq  \limsup_{n\to\infty}
\left( \sum_{i=1}^{N}  \P
(\tau_1>\delta\sqrt{n})+ \sum_{i=N+1}^\infty  \P
(\tau_1>\frac{\delta \sqrt{i}}{ \sqrt{T}}) \right) \leq \epsilon\,.
\end{equation*} As $\epsilon >0 $ was arbitrary, one concludes from the
last limit and (\ref{step1}) that  (\ref{step0}) holds. 
\end{proof}

\medskip

\newcommand{\etalchar}[1]{$^{#1}$}

\end{document}